\documentclass[12pt, leqno]{article}

\usepackage{amsmath,amsthm}
\usepackage{amssymb}

\usepackage{enumerate}

\usepackage{graphicx}

\newtheorem{thm}{Theorem}[section]
\newtheorem{cor}[thm]{Corollary}

\newtheorem{prop}[thm]{Proposition}

\theoremstyle{definition}
\newtheorem{definition}[thm]{Definition}

\numberwithin{equation}{section}

\usepackage{amsfonts}
\usepackage{mathrsfs}
\usepackage{enumerate}
\usepackage{longtable}

\usepackage{url}
\begin{document}



\title{Computations of Galois Representations Associated to Modular Forms of Level One}

\author{Peng Tian\\
Department of Mathematics\\
Nanjing University\\
210093, Nanjing, P. R. China\\
and\\
Dipartimento di Matematica\\
Universit\`{a} di Roma \textquotedblleft Tor Vergata''\\
00133, Rome, Italy.\\
E-mail: tianpeng.china@gmail.com}

\date{}

\maketitle


\renewcommand{\thefootnote}{}

\footnote{2010 \emph{Mathematics Subject Classification}: 11-04, 11Fxx, 11G30, 11Y40.}

\footnote{\emph{Key words and phrases}: modular Galois representations, modular forms, modular curves, Jacobian, Ramanujan's tau function, polynomials.}

\renewcommand{\thefootnote}{\arabic{footnote}}
\setcounter{footnote}{0}

\begin{abstract}

We propose an improved algorithm for computing mod $\ell$ Galois representations associated to a cusp form $f$ of level one. The proposed method allows us to explicitly compute the case with $\ell=29$ and $f$ of weight $k=16$, and the cases with $\ell=31$ and $f$ of weight $k=12,20, 22$. All the results are rigorously proved to be correct.

As an example, we will compute the values modulo $31$ of Ramanujan's tau function at some huge primes up to a sign. Also we will give an improved higher bound on Lehmer's conjecture for Ramanujan's tau function.
\end{abstract}

\section{Introduction}
In 1995, Ren\'{e} Schoof asked Bas Edixhoven that: given a prime number $p$, can one compute Ramanujan's tau function $\tau(p)$ defined by
 \[\varDelta(z)=q\prod_{n\geq 1}(1-q^{n})^{24}=\sum_{n}\tau(n)q^{n}\]
 in time polynomial in $\log p$?
 
 In the book \cite{book}, S. J. Edixhoven, J.-M. Couveignes, R. S. de Jong and F. Merkl give an affirmative answer. They generalize Schoof's algorithm \cite{Schoof85} and show that
 
 \textit{There exists a deterministic algorithm that on input a prime number $p$ computes $\tau(p)$ in time polynomial in $\log p$.} 

 Ramanujan observed the remarkable property of $\tau(p)$:
\begin{center}
  $|\tau(p)|\leq2p^{11/2}$ for prime $p$,
\end{center}
 which was proved by P. Deligne. In fact, he \cite{Deligne71} shows that there exists a continuous semi-simple representation
  \begin{center}
   $\rho_{\varDelta,\ell}: Gal(\overline{\mathbb{Q}}|\mathbb{Q}) \rightarrow GL_{2}(\overline{\mathbb{F}}_{\ell})$.
  \end{center}
 This representation is unique up to isomorphism and it has the property that for primes $p$ not dividing $N\ell$ one has
 \[\tau(p)\equiv \mathrm{tr}(\rho_{\varDelta,\ell}(\mathrm{Frob}_{p})) \mod \ell . \]
 
 In \cite{book}, Edixhoven and Couveignes give a polynomial time algorithm to compute the modular Galois representation and thus the value modulo $\ell$ of Ramanujan's tau function at $p$. Then combining with the property $|\tau(p)|\leq2p^{11/2}$ for primes $p$ and the Chinese remainder theorem one can compute $\tau(p)$. It is well known that the representation appears in the group of $\ell$-torsion points of the Jacobian variety of the modular curve $X_1(\ell)$. If the genus $g$ of  $X_{1}(\ell)$ is equal to $1$, the question boils down to the case of an elliptic curve, which has been solved by Schoof's algorithm. 
 
 Since the Galois representation $\rho_{\Delta,\ell}$ is $2$-dimensional,
 the fixed field of ${\rm ker}(\rho_{\Delta,\ell})$ can be described as the splitting field of a certain polynomial
 $P_{\Delta,\ell} \in \mathbb{Q}[x]$ of degree $\ell^2-1$. Moreover, the associated projective representation 
  can be described as the splitting field of a certain polynomial
 $\tilde P_{\Delta,l}\in \mathbb{Q}[x]$ of degree $\ell+1$.
 
In general, all the discussions above hold for modular forms with level $1$.
 
 Unfortunately the algorithm described in \cite{book} is difficult to implement. J. Bosman \cite{bosmanpaper} used this algorithm to approximately evaluate $\tilde P_{f,\ell}$ of mod $\ell$ Galois representations associated to modular forms $f$ of level 1 and of weight $k\le 22$, with $\ell \leq 23$. But since the required precision in the calculations grows quite rapidly with $\ell$, Bosman did not compute more cases.
 
 In this paper we present an improvement in case gcd$(k-2,l+1)>2$. In these cases there is a modular curve $X_{\varGamma}$ with $\varGamma_{1}(\ell)\lneq \varGamma \leq \varGamma_{0}(\ell)$ with the property that the $2$-dimensional Galois representation is a subrepresentation of the $\ell$-torsion points of the Jacobian of $X_{\varGamma}$. Therefore we can do the computations  with the Jacobian of $X_{\varGamma}$ rather than the Jacobian of $X_1(\ell)$ that Bosman used. Since the genus of $X_{\varGamma}$ is smaller than that of $X_1(\ell)$, the required precision is smaller and the computation  is more efficient. This allows us to deal with cases that were inaccessible  by Bosman's original algorithm.
 
As an example, we succeed to compute the mod $31$ Galois representation associated to discriminant modular form $\varDelta$. For $\ell=29$ and $31$, we also compute the mod $\ell$ Galois representation associated to the unique normalised cusp forms of level 1 and weights 16, 20 and 22. The correctness of each $\tilde P_{f,l}$ is then verified by an application of Serre's conjecture, proved by Khare-Wintenberger \cite{serreconjfull}.

We compute the values modulo $31$ of Ramanujan's $\tau$ function at some huge primes up to a sign.  As a consequence we can verify Lehmer's conjecture up to a large bound. More precisely, we show that
$$
 \tau(n)\not=0,\qquad\hbox{for all $n< 982149821766199295999$}
$$
This improves  Bosman's  bound by a factor approximately equal to $43$.

\section{Outline of the Algorithm}

Let $N$ be a positive integer. The congruence subgroup $\varGamma _{1}(N)$ of level $N$ is
\begin{center}
$\varGamma _{1}(N) = \left\lbrace 
 \begin{pmatrix}
 a & b\\
 c & d
 \end{pmatrix}  \in SL(2,\mathbb{Z}) \ \mid \ c \equiv 0 \ mod \ N , \  \ \ a\equiv b\equiv 1 \ mod \ N
 \right\rbrace $
\end{center}
Let $k\geq 2$ be an integer. Let $f = \sum_{n>0} a_{n} (f) q^{n} \in S_{k}(\varGamma_{1}(N) $ be a newform of weight $k$ and level $N$. Let $ \varepsilon $ be its nebentypus character. Let $K_f$ be the number field which is obtained by adjoining all coefficients $ a_{n} $ of the $q$-expansion $f$ to $\mathbb Q$. Let $\ell$  be a prime number. Let $\lambda$ be a prime of $K_{f}$ lying over $\ell$. Denote by $K_{f,\lambda}$ the completion of $K_f$ at the prime $\lambda$. Then thanks to Deligne, we know that there exists an irreducible representation associated to $f$
 \begin{center}
  $\rho_{f,\lambda}:Gal(\overline{\mathbb{Q}}|\mathbb{Q}) \rightarrow GL_{2}(K_{f,\lambda})$,
 \end{center} 
 that is unramified outside $N\ell$. Furthermore, for all primes $p \nmid N\ell$, the characteristic polynomial of the representation at Frobenius element $Frob_{p}$ is $x^2-a_{p}(f)x+ \varepsilon(p)p$. It is possible to reduce the representation modulo $\lambda$. We have the following well known theorem:
 
\begin{thm}
f $\in S_{k}(N, \varepsilon)$ be a newform. Let $\lambda$ be as above and let $\mathbb{F}_{\lambda}$ denote the residue field of $\lambda$. Then there exists a continuous
semi-simple representation
 \begin{center}
  $\rho_{f,\lambda}: Gal(\overline{\mathbb{Q}}|\mathbb{Q}) \rightarrow GL_{2}(\mathbb{F}_{\lambda})$.
 \end{center}
 that is unramified outside $N\ell$, and for all primes $p\nmid N\ell$ the characteristic polynomial of  $\rho_{f,\lambda} (Frob_{p}) $ satisfies
\begin{equation} \label{charpol}
charpol( \rho_{f,\lambda} (Frob_{p}))\equiv x^2-a_{p}(f)x+ \varepsilon(p)p^{k-1} \mod{\lambda}.
\end{equation}
Moreover, $\rho_{f,\lambda}$ is unique up to isomorphism.
\end{thm} 

The discriminant modular form is given by
 \[\varDelta(z)=q\prod_{n\geq 1}(1-q^{n})^{24}=\sum_{n}\tau(n)q^{n},\]
and its Fourier coefficients define the Ramanujan's tau function $\tau(n)$. 
 
 In the book \cite{book}, S. Edixhoven and J.-M. Couveignes generalize Schoof's algorithm \cite{Schoof85} and show that 
 
\textit{There exists a deterministic  algorithm that computes the $\mathrm{mod} \ \ell$ Galois representation associated to level one modular forms in time polynomial in $\ell$.} 

Since we have the congruence relation
\[\tau(p)\equiv \mathrm{tr}(\rho_{\varDelta,\ell}(\mathrm{Frob}_{p})) \mod \ell,  \]
this algorithm can be used to compute $\tau(p)$ mod $\ell$ in time polynomial in $\log p$ and $\ell$.

Fix a prime number $\ell$ and let $\lambda$ be a prime lying over $\ell$. The residue field of the ring of integers of $\overline{\mathbb Q}_\ell$ is isomorphic to $\overline{\mathbb{F}}_\ell$. And then since $\mathbb F_{\lambda}\subset \overline{\mathbb F}_\ell$, we can view our representation $\rho_{f,\lambda}$ as taking values in $GL_{2}(\overline{\mathbb{F}}_\ell)$. In \cite[Theorem 2.2]{report}, the author shows that 
if $2 < k \leq \ell+1$ and $\rho_{f,\lambda}$ is ireducible, then there is a newform $ f_{2}\in S_{2}(\varGamma_{1}(N\ell))$, together with a prime $\lambda_{2}$ lying over $\ell$ of the coefficient field $K_{f_{2}}$, such that $\rho_{f_{2},\lambda_{2}}$ is isomorphic to $\rho_{f,\lambda}$. Therefore, for any $p \nmid N\ell$,  the matrices $\rho_{f,\lambda} (Frob_{p})$ and $\rho_{f_{2},\lambda_{2}} (Frob_{p})$ have the same characteristic polynomial in $\overline{\mathbb{F}}[x]$. This allows Edixhoven and Couveignes to reduce the questions to the weight 2 cases.

Now we suppose that $\rho_{f,\lambda}$ is a mod $\ell$ Galois representation associated to a newform $f\in S_{2}(\varGamma_1(\ell))$ with character $\varepsilon$. Let $X_{1}(\ell)$ be the modular curve associated to $\varGamma_{1}(\ell)$ and $J_{1}(\ell)$ denote its Jacobian. Denoting $\mathbb T$ the Hecke algebra generated by the diamond and Hecke operators over $\mathbb Z$, i.e. $\mathbb T=\mathbb{Z}[  T_{n}, \langle n\rangle : n \in \mathbb Z_{+} \ \mathrm{and} \ (n,\ell)=1]$, then $\mathbb{T}\subset$ End$J_{1}(\ell)$ and we have a ring homomorphism $\theta: \mathbb{T} \rightarrow \mathbb F_{\lambda} ,$ given by $\langle d \rangle \mapsto \varepsilon(d)$ and $T_{n} \mapsto a_{n}(f)$.
Let $\mathfrak m$ denote the kernel of $\theta$ and put 
\begin{displaymath} 
 V_{\lambda}=J_{1}(\ell)(\overline{\mathbb{Q}})[\mathfrak{m}]=\{x \in J_{1}(\ell)(\overline{\mathbb{Q}}) \ | \ tx=0 \ \mathrm{for}  \ \mathrm{all} \ t \ \mathrm{in} \ \mathfrak m \}.
\end{displaymath} 
This is a 2-dimensional $\mathbb{T}/\mathfrak{m}$-linear subspace of $J_{1}(\ell)(\overline{\mathbb{Q}})[\ell]$ and the semisimplification of the representation
$$
\rho: Gal (\mathbb{\overline{Q}}/\mathbb{Q}) \rightarrow \mathrm{Aut}(V_\lambda)
$$
is isomorphic to $\rho_{f,\lambda}$(see \cite[Section 3.2 and 3.3]{ribetstein}).

 If \# $\mathbb{T}/\mathfrak{m}=\ell$, then the  fixed field of $\rho_{f,\lambda}$ is naturally  the splitting field of a 
suitable polynomial $P_{f,\lambda}\in \mathbb Q[X]$ of degree $\ell^2-1$. More precisely,  we can take
\begin{equation}\label{polynomial}
P_{f,\lambda}(x) = \prod_{P\in V_\lambda-\{0\}}(x-h(P))
\end{equation}
for some suitable function $h$ in the function field of $X_1(\ell)$. Here $h(P)$ has the following meaning.
If $g$ is the genus of $X_1(\ell)$, then we can write each divisor $P\in V_\lambda-\{0\}$ as $\sum_{i=1}^g (P_i) -gO)$
for certain points $P_i$ on $X_1(\ell)$. We put $h(P)=\sum_{i=1}^g h(P_i)$. 

Composed with the canonical projection map $GL_{2}(\mathbb{F}_{\lambda})\rightarrow PGL_{2}(\mathbb{F}_{\lambda})$, the representation $\rho_{f,\lambda}$ gives a projective representation $\tilde{\rho}_{f,\lambda}: Gal(\overline{\mathbb{Q}}|\mathbb{Q}) \rightarrow PGL_{2}(\mathbb{F}_{\lambda})$.

Since the  projective line $\mathbb P(V_{\lambda})$ has $\ell+1$ points,  the  fixed field of $\tilde{\rho}_{f,\lambda}$ is naturally  the splitting field of a 
suitable polynomial $\tilde P_{f,\lambda}\in Q[X]$ of degree $\ell+1$. More precisely,  we can take
 \begin{equation}\label{projectivepolynomial}
\tilde P_{f,\lambda}(x) = \prod_{L\subset \mathbb P(V_{\lambda})}(x-\sum_{P\in L-\{0\}}h(P)).
  \end{equation}
  
J. Bosman first uses a complex approximation approach to compute the points in $V_{\ell}$ over $\mathbb{C}$ and then from these computed points evaluates approximately $\tilde{P}_{f,\ell}$. In the end, he explicitly computes mod $\ell$ Galois projective representations associated to modular forms of level 1 and weight up to 22, with $\ell \leq 23$. For details, we refer to \cite[Chapter 2]{bosmanthesis}.

\section{Galois Theory of Modular Curves}

Let $\ell$ be a prime number and $\varGamma$ be a congruence subgroup of level $\ell$. Let $\mathfrak h$ denote the upper half plane and $X_{\varGamma}=(\varGamma\backslash\mathfrak h)\cup(\varGamma\backslash(\mathbb{Q}\cup \infty))$ be the modular curve for $\varGamma$. The function field of the modular curve $X_{\varGamma}$ is denoted by $\mathbb C(X_{\varGamma})$. Then from \cite[Section 7.5]{diamondmodularforms} we know that the function field extension $\mathbb C(X(\ell))|\mathbb C(X(1))$ is Galois with Galois group
\begin{displaymath}
\mathrm{Gal}(\mathbb C(X(\ell))|\mathbb C(X(1)))\cong SL_{2}(\mathbb Z/\ell\mathbb Z)/\{\pm I\},
\end{displaymath}
and the extension $\mathbb C(X_{1}(\ell))|\mathbb C(X_{0}(\ell))$ is Galois with Galois group 
\begin{equation}\label{galoisgroup}
\mathrm{Gal}(\mathbb C(X_{1}(\ell))|\mathbb C(X_{0}(\ell)))\cong \{\pm I\}\varGamma_{0}/\{\pm I\}\varGamma_{1}\cong(\mathbb{Z}/\ell\mathbb{Z})^{\ast}/\{\pm 1\}.
\end{equation}
\begin{definition}
Let $\varGamma_{1} \leq \varGamma_{2}$ be congruence subgroups. The natural morphism $X_{\varGamma_{1}}\rightarrow X_{\varGamma_{2}}$ is said to be \emph{Galois} if the extension $\mathbb C(X_{\varGamma_{1}})|\mathbb C(X_{\varGamma_{2}})$ is Galois. The \emph{Galois group of $X_{\varGamma_{1}}\rightarrow X_{\varGamma_{2}}$} is defined to be Gal(($\mathbb C(X_{\varGamma_{1}})|\mathbb C(X_{\varGamma_{2}})$).
\end{definition}

This allows us to speak of the Galois theory of modular curves over $\mathbb{C}$ via the Galois theory of their function fields. Let $G=$Gal$(\mathbb C(X_{\varGamma})|\mathbb C(X_{0}(\ell)))$. Since the meromorphic differentials of the modular curve $X_{\varGamma}$ form a $1$-dimensional vector space over  $\mathbb{C}(X_{\varGamma})$ generated by d$f$ for a non-constant function $f\in \mathbb C(X_{\varGamma})$ that is $\varGamma$-invariant, the differentials space is isomorphic to $\mathbb{C}(X_{\varGamma})$ as $G$-module.

\section{Algorithm for Our Cases}\label{sec:algorithm}

In this section we will explain how to compute the polynomial $ P_{f,\lambda}$ in (\ref{polynomial}), the splitting field of which is the fixed field of the Galois representations $\rho_{f,\lambda}$ associated to a modular form $f$ of level $1$. All the discussions in this section also hold for the case of projective polynomial $ \tilde P_{f,\lambda}$ in (\ref{projectivepolynomial}). As explained in Section 2, our main task is then to compute the 2-dimensional $\mathbb{F}_\lambda$-linear space $V_{\lambda}$. We will do this as Bosman except that we work with a modular curve that sometimes has smaller genus than $X_1(\ell)$.

\subsection{Finding modular curves} \label{step1}
Let $k>0 $ be an even  integer and let $\ell$ be a prime number with $k \leq \ell+1$. Let  $f \in S_{k}( SL(2,\mathbb{Z}))$ be a newform of level $1$. In general, the modular curve to realize the representation $\rho_{f,\lambda}$ is $X_{1}(\ell)$ which has genus $(\ell-5)(\ell-7)/24$, but we have

\begin{prop} \label{myweight2}
Let $k>0 $ be an even  integer and f $\in S_{k}( SL(2,\mathbb{Z}))$ be a newform of level $1$ and weight $k$. Let $ \ell \geq k-1$ be a prime number  and $\lambda$ be a prime lying over $\ell$. Let $\varGamma$ be the unique group  
$$
\varGamma_1(\ell) \subset \varGamma \subset \varGamma_0(\ell)
$$
with $[\varGamma:\varGamma_1(\ell)] = \frac{1}{2}gcd(k-2,\ell-1)$. Then there exists a newform $f_2\in S_2(\varGamma)$ and a prime $\lambda_2$ lying over $\ell$  in the field $K_{f_2}$ such that $\rho_{f,\lambda}$ is isomorphic to $\rho_{f_2,\lambda_2}$.
\end{prop}

\begin{proof}
It follows from \cite[Theorem 2.2]{report} that there exists $f_{2}\in S_{2}(\varGamma_{1}(\ell))$ and a prime $\lambda_{2}|\ell$ such that $\rho_{f,\lambda}$ is isomorphic to $\rho_{f_{2},\lambda_{2}}$. Since the character of $f$ is trivial in our case, for any $p \nmid \ell$, by (\ref{charpol}) we have the equalities in $\overline{\mathbb{F}}$:
\begin{equation} \label{tracedet}
a_{p}(f_{2}) \ = \ a_{p} (f)  \ \  \ \  \ and \ \  \ \ \ \ \varepsilon_{2}(p) =  p^{k-2} 
\end{equation}
Here $\varepsilon_{2}$ is the nebentypus character of $f_{2}$, which is a Dirichlet character of the cyclic group $(\mathbb{Z}/\ell \mathbb{Z})^{\ast}$. Let $\omega$ be the cyclotomic character and then it follows from the  second equation in (\ref{tracedet}) that $\varepsilon_{2}=\omega^{k-2}$.

By (\ref{galoisgroup}), the map $X_{1}(\ell)\rightarrow X_{0}(\ell)$ is Galois and its Galois group is $(\mathbb{Z}/\ell\mathbb{Z})^{\ast}/\{\pm 1\}$. Now let $H$ denote the normal subgroup ker$(\omega^{k-2})/\{\pm 1\}$ of $(\mathbb{Z}/\ell\mathbb{Z})^{\ast}/\{\pm 1\}$. By the Galois theory of function fields of modular curves, we have an intermediate curve $X$ of $X_{1}(\ell)\rightarrow X_{0}(\ell)$ such that the Galois group of $X_{1}(\ell)\rightarrow X$ is $H$. Let $\varphi$ denote the surjection:
\begin{equation}\label{map}
\varphi :  \ \varGamma_{0}(\ell) \twoheadrightarrow (\mathbb{Z}/\ell\mathbb{Z})^{\ast} , \ \ \ \ \ \ \
\begin{pmatrix} a & b\\ c & d \end{pmatrix} \rightarrow \overline{d} 
\end{equation}
whose kernel is $\varGamma_{1}(\ell)$. Let $\varGamma_{H}$ be the preimage of $\{\pm 1\}H$ under $\varphi$. Then we have $X=X_{\varGamma_{H}}$ and ker$(\varphi)\subseteq\varGamma_H$, since $\#H=\frac{1}{2}$gcd$(k-2,\ell-1)$.

To complete the proof we only need to check that $f_{2} \in S_{2}(\varGamma_{1}(\ell))$ also lies in $S_{2}(\varGamma_{H})$. In fact, for all $ \gamma=\begin{pmatrix} a&b\\ c&d\end{pmatrix} \in \varGamma_{H}$, it follows from the definition of $\varGamma_{H}$ that $\varphi(\gamma)$ is in ker$(\omega^{k-2})$ and thus $f_2|_{2}\gamma=\omega^{k-2}(\varphi(\gamma))f_2=f_2$, which implies $f_{2} \in S_{2}(\varGamma_{H})$. 
\end{proof}
Once we find a congruence subgroup $\varGamma_{H}$ of level $\ell$ with $\varGamma_1(\ell) \subset \varGamma \subset \varGamma_0(\ell)$ , such that the associated newform $f_{2}\in S_{2}(\varGamma_{1}(\ell)) $ lies in $ S_{2}(\varGamma_{H})$, then $X_{\varGamma_{H}}$ can be taken as the modular curve to realize the representations. The proof implies that $X_{1}(\ell)\rightarrow X_{\varGamma_{H}}$ is Galois with Galois group $H$.

The gcd$(\ell-1,k-2)$ cannot be larger than $\ell-1$ or $k-2$. If it is equal to $\ell-1$,
the character $\omega^{k-2}$ is trivial and the representation is
actually a subrepresentation of the $\ell$-torsion of the Jacobian of $X_0(\ell)$.
This happens for instance for $k=12$ and $\ell=11$ reducing the computation to 
a calculation on the genus $1$ curve $X_0(11)$. If it is equal to $k-2$, we have $ \ell \geq k-1$ with $\ell \equiv 1 \pmod{k-2}$. This happens for instance for  $k=12$ and $\ell=31$.

\subsection{Realization of $\rho$ in the Jacobian of a modular curve} \label{sec:torsion}

Let $k>0 $ be an even  integer. Suppose $H$=ker$(\omega^{k-2})/\{\pm 1\}$. In this subsection, we assume that $ \ell \geq k-1$ is a prime number and $f\in S_{2}(\varGamma_{1}(\ell))$ has character $\omega^{k-2}$ and then lies in $S_{2}(\varGamma_{H})$ where $\varGamma_{H}$ corresponds to $H$ via $\varphi$ in (\ref{map}). Let $X_{\varGamma_{H}}$ be the modular curve of the subgroup $\varGamma_{H}$ and denote $J_{\varGamma_{H}}$ its Jacobian. Then $X_{1}(\ell)\rightarrow X_{\varGamma_{H}}$ is Galois with Galois group $H$. As discussed in Section 3, the meromorphic differential space over $X_{1}(\ell)$ is isomorphic to $\mathbb{C}(X_{1}(\ell))$ as Gal$(\mathbb C(X_{1}(\ell))|\mathbb C(X_{0}(\ell)))$-module and from this it follows that the holomorphic differential space $\varOmega^{1}_{hol}(X_{\varGamma_{H}})$ is the $H$-invariant part of $\varOmega^{1}_{hol}(X_{1}(\ell))$. By taking duals of these spaces, we get
\begin{displaymath}
J_{\varGamma_{H}}(\overline{\mathbb{Q}})[\ell]=J_{1}(\ell)(\overline{\mathbb{Q}})[\ell]^{H}:=\{x \in J_{1}(\ell)(\overline{\mathbb{Q}})[\ell]  \ | \ \sigma(x)=x, \mathrm{ \ for \ all \ } \sigma \in H \}.
\end{displaymath}

As discussed in Section 2, the representation associated to $f$ is a subrepresentation of the $\ell$-torsion points of $J_{1}(\ell)$. However, in our cases one can work with $J_{\varGamma_{H}}$ instead of $J_{1}(\ell)$:
\begin{prop} \label{mytorsion}
The torsion space $V_{\lambda}$ is a $2$-dimensional subspace of $J_{\varGamma_{H}}(\overline{\mathbb{Q}})[\ell]$.
\end{prop}

\begin{proof}It follows from the definition of $H$ that each $\sigma\in H$ acts on $J_{1}(\ell)(\overline{\mathbb{Q}})[\ell]$ the same as a diamond operator $\langle d \rangle$ for some $d\in (\mathbb{Z}/\ell\mathbb{Z})^{\ast}$ with $d^{k-2}=1$. This implies that $\sigma-id$ is an element of $\mathfrak{m}$=ker$(\theta)$ and thus $V_{\lambda}\subset J_{1}(\ell)(\overline{\mathbb{Q}})[\ell]^{H}=J_{\varGamma_{H}}(\overline{\mathbb{Q}})[\ell]$.
\end{proof}

\subsection{Description of the computations}\label{algorithm}

Now we show how to explicitly compute the polynomial
\begin{displaymath}
P_{f,\lambda}(x) = \prod_{P\in V_\lambda-\{0\}}(x-h(P)).
\end{displaymath}
First of all, $J_{\varGamma}(\mathbb{C})[\ell]$ can be described in terms of modular symbols by the isomorphisms 
\[J_{\varGamma}(\mathbb{C})[\ell]\cong H_{1}(X_{\varGamma},\mathbb{F}_{\ell})\cong \mathbb{S}_{2}(\varGamma)\otimes \mathbb{F}_{\ell}. \]
 Let $g$ be the genus of $J_{\varGamma}$. Taking a basis $f_{1},...,f_{g}$ of $S_{2}(\varGamma)$, we can compute the period lattice $\varLambda\subset \mathbb{C}^{g}$ by integrating $(f_{1},...,f_{g})$ along elements of $H_{1}(X_{\varGamma},\mathbb{F}_{\ell})$. Let $\mathbb{T}'$ be the the Hecke algebra over $S_{2}(\varGamma)$. Since the action of $\mathbb{T}'\subset$End($J_{\varGamma}$) on $\mathbb{S}_{2}(\varGamma)$ can be numerically computed \cite{Stein07}, we thus obtain approximations of the torsion points of $V_{\lambda}\subset \frac{1}{\ell}\varLambda/\varLambda$. Then using the Newton iteration approximation method, we can find torsion divisors via the Abel-Jacobi map and finally compute the polynomial in (\ref{polynomial}).

This approximation method requires very high precision even when $\ell$ is quite small. But since the precision depends on the dimension of the Jacobian $J_{\varGamma}$, replacing by $J_{\varGamma}$ whose dimension is smaller than $J_{1}(\ell)$ reduces a large number of calculations and therefore we can compute the cases for larger $\ell$.

\section{Examples} \label{sec:result}

For $k=12,14,16,18,20$ and $22$, let $\Delta_k$ denote the unique cusp form of level $1$ and weight $k$.
In \cite{bosmanpaper},  Bosman computed the modular projective polynomials  $\tilde P_{\varDelta_{k},\ell}$ for several values of $\ell$ and $k$. We add a few more polynomials to this list using the algorithm described in Section  \ref{sec:algorithm}.
  
We first give a list of $(k,\ell)$ with gcd$(k-2,\ell-1)>2$ for which we have computed the polynomials $\tilde{P}_{\varDelta{_k},\ell}$  together with the dimensions of $J_1(l)$ and $J_{\Gamma_H}$:

 \begin{table}[htc]
 \begin{tabular} {|c|c|c|c|}
  \hline 
  ($k$,$\ell$) & gcd$(k-2,\ell-1)$ & dimension of $J_{1}(\ell)$ & dimension of $J_{\varGamma_{H}}$  \\  \hline
  (12, 31) & 10 & 26 & 6 \\ \hline
  (16, 29) & 14 & 22 & 4  \\  \hline
  (20, 31) & 6  & 26 & 6  \\  \hline
  (22, 31) & 10 & 26 & 6  \\  \hline
  \end{tabular} \\
 \end{table}

 The corresponding polynomials are

\begin{longtable}{|c|c|}

  \hline  & \\ 
  ($k$,$\ell$) &  $\tilde{P}_{\varDelta_{k},\ell}$   \\ [10pt] \hline 
 (12, 31) &  
  \begin{minipage}[t]{0.8\textwidth} 
  $x^{32} - 4x^{31} - 155x^{28} + 713x^{27} - 2480x^{26} + 9300x^{25} - 5921x^{24} + 24707x^{23} + 127410x^{22} - 646195x^{21} + 747906x^{20} - 7527575x^{19} + 4369791x^{18} - 28954961x^{17} - 40645681x^{16} + 66421685x^{15} - 448568729x^{14} + 751001257x^{13} - 1820871490x^{12} + 2531110165x^{11} - 4120267319x^{10} + 4554764528x^{9} - 5462615927x^{8} + 4607500922x^{7} - 4062352344x^{6} + 2380573824x^{5} - 1492309000x^{4} + 521018178x^{3} - 201167463x^{2} + 20505628x - 1261963$ 
  \end{minipage} \\  \hline
  (16,29)&
 \begin{minipage}[t]{0.8\textwidth} 
  $x^{30} - 13x^{29} + 116x^{28} - 899x^{27} + 6003x^{26} - 33002x^{25} + 142158x^{24} - 437871x^{23} + 599981x^{22} + 3161522x^{21} - 30157709x^{20} + 149069425x^{19} - 545068137x^{18} + 1602112888x^{17} - 3929042061x^{16} + 8240756348x^{15} - 15020495335x^{14} + 23992472995x^{13} - 33394267804x^{12} + 40034881756x^{11} - 40888329774x^{10} + 35730188833x^{9} - 27316581262x^{8} + 17713731976x^{7} - 7068248851x^{6} - 1463296732x^{5} + 4054490087x^{4} - 2555610007x^{3} + 2573924261x^{2} + 2363203645x - 261910751$ 
 \end{minipage} \\ \hline
   (20,31)&
  \begin{minipage}[t]{0.8\textwidth} 
$x^{32} - 4x^{31} - 62x^{30} + 558x^{29} - 248x^{28} - 23560x^{27} + 143499x^{26} + 59489x^{25} - 4280108x^{24} + 17190864x^{23} + 12517459x^{22} - 344750256x^{21} + 1225662500x^{20} - 278789479x^{19} - 14790203106x^{18} + 64357190741x^{17} - 83774789980x^{16} - 406418167694x^{15} + 2480836111912x^{14} - 5273524311353x^{13} - 3257558862543x^{12} + 54285321863574x^{11} - 162450534558477x^{10} + 197719989210108x^{9} + 250865100757790x^{8} - 1714511602191278x^{7} + 4206562171750919x^{6} - 6661579151098950x^{5} + 7460752526582377x^{4} - 5959749341609879x^{3} + 3269911760551427x^{2} - 1113936554991727x + 178725601175511$
  \end{minipage} \\ \hline
   (22,31)&
    \begin{minipage}[t]{0.8\textwidth} 
  $x^{32} - 3x^{31} - 124x^{30} + 651x^{29} + 5797x^{28} - 44020x^{27} - 46593x^{26} + 1523309x^{25} - 4960682x^{24} - 28562129x^{23} + 205283395x^{22} + 345367838x^{21} - 3865963779x^{20} - 5281917640x^{19} + 35629245810x^{18} + 95827452774x^{17} + 227525150938x^{16} - 1735983387875x^{15} - 9952753525850x^{14} + 15867354189588x^{13} + 146446287180279x^{12} - 99789981007214x^{11} - 1135328992145553x^{10} - 171825071648506x^{9} + 7446294546204081x^{8} + 294530833190147x^{7} - 24397472702475140x^{6} - 9976638213111902x^{5} + 61714590456038129x^{4} + 16902762581347117x^{3} - 13833080015551423x^{2} - 202960986205176103x + 187532019539254309$
    \end{minipage} \\ \hline
   
    \caption{Polynomials}
\label{table:polynomials}
  \end{longtable}
  
  The computations done to obtain these polynomials required a precision of about 
  4200 bits for $\ell=31$ and 3500 bits for $\ell=29$. The calculations ahev been done in SAGE \cite{sage}.
  They took about $10$ days for each of the cases with $\ell=31$ and one week for the case $\ell=29$. 
  The polynomial $\tilde{P}_{\varDelta_{12},31}$ has also been obtained by Zeng \cite{zeng}. His method 
  avoids the high precision computations and is based on $p$-adic computations.  Mascot \cite{mascot}
  claims to have computed a polynomial $P_{\varDelta_{12},29}$, but unfortunately he has not
  provided us with the polynomial.

It is difficult to rigorously prove that the computations have been 
done with sufficient accuracy and that therefore the results are correct.
However, once the polynomial is computed, one can verify that it is correct using
Serre's conjecture.

Let $\ell$ be a prime. A Galois representation $\rho : Gal(\overline{\mathbb{Q}}|\mathbb{Q}) \rightarrow GL_{2}(\overline{\mathbb{F}}_{\ell})$
 has a Serre level $N(\rho)$ and a Serre weight $k(\rho)$. See \cite{serreconj} for Serre's definition and \cite{basweight} for a reformulation. Then we have the following famous Theorem which has been fully proved by C. Khare and J. P. Wintenberger in 2008:

\begin{thm}[Serre's Conjecture]
Let $\ell$ be a prime and let $\rho$: $Gal(\overline{\mathbb{Q}}|\mathbb{Q}) \rightarrow GL_{2}(\overline{\mathbb{F}}_{\ell})$ be a representation that is irreducible and odd. Then there exists a newform $f$ of level $N(\rho)$ and weight $k(\rho)$ and a prime $\lambda$ of $K_{f}$ above $\ell$ such that $\rho$ is isomorphic to $\overline{\rho}_{f,\lambda}$.
\end{thm}
\begin{proof}
See \cite{serreconjfull}
\end{proof}

Now we have
 \begin{prop} \label{mypolynomial}
 For each pair $(k,\ell)$ in Table \ref{table:polynomials}, we denote $\varDelta_{k}$ the normalized newform of weight $k$ and level $1$. Then the polynomial $\tilde{P}_{\varDelta_{k},\ell}$ in the Table \ref{table:polynomials} is irreducible. The Galois group of its splitting field is isomorphic to $PGL_{2}(\mathbb{F}_{\ell})$. Moreover, a subgroup of $Gal(\mathbb{\overline{Q}|Q})$ fixing a root of $\tilde{P}_{\varDelta_{k},\ell}$ corresponds via $\tilde{\rho}_{\varDelta,\ell}$ to a subgroup of $PGL_{2}(\mathbb{F}_{\ell})$ fixing a point of $\mathbb{P}^{1}(\mathbb{F}_{\ell})$.
 \end{prop}
\begin{proof}
First, the algorithm in \cite[Algorithm 6.1]{galoisgroupcomputation} which has been implemented in MAGMA \cite{magma} was used to compute the Galois group $Gal(\tilde{P}_{\varDelta_{k},\ell})$ of the polynomials in Table \ref{table:polynomials} as a permutation group acting on the roots. The practical calculations can be done in several seconds. It provides us with an isomorphism
\begin{equation} \label{isogalandpgl31}
Gal(\tilde{P}_{\varDelta_{k},\ell})\cong PGL_{2}(\mathbb{F}_{\ell}).
\end{equation}

Then we have a projective representation $\tilde{\rho}_{k,\ell}: Gal(\bar{\mathbb{Q}}/\mathbb{Q})\twoheadrightarrow PGL_{2}(\mathbb{F}_{\ell})$ by composing the canonical map $Gal(\bar{\mathbb{Q}}/\mathbb{Q})\twoheadrightarrow  Gal(\tilde{P}_{\varDelta_{k},\ell})$ with the isomorphism in (\ref{isogalandpgl31}). Since the group $PGL_{2}(\mathbb{F}_{\ell})$ has no outer automorphisms, up to isomorphism $\tilde{\rho}_{k,\ell}$ is uniquely determined by $\tilde{P}_{\varDelta_{k},\ell}$.

We denote $K_{k,\ell}:=\mathbb{Q}[x]/(\tilde{P}_{\varDelta_{k},\ell})$ the number field defined by $\tilde{P}_{\varDelta_{k},\ell}$ and the integer ring of $K_{k,\ell}$ is denoted by $\mathcal{O}_{k,\ell}$.

Let $G$ be a subgroup of $Gal(\mathbb{\overline{Q}|Q})$ fixing a root of $\tilde{P}_{\varDelta_{k},\ell}$. By the canonical map $Gal(\overline{\mathbb{Q}}/\mathbb{Q})\twoheadrightarrow Gal(\tilde{P}_{\varDelta_{k},\ell})$, the group $G$ corresponds to a subgroup of $Gal(\tilde{P}_{\varDelta_{k},\ell})$ of index $[K_{k,\ell} : \mathbb{Q}]=$deg$(\tilde{P}_{\varDelta_{k},\ell})=\ell+1$, and thus the image of $G$ via $\tilde{\rho}_{k,\ell}$ is a subgroup of $PGL_{2}(\mathbb{F}_{\ell})$ of index $\ell+1$, which by \cite[Lemma 7.3.2]{book} is the stabiliser subgroup of a point in $\mathbb{P}^{1}(\mathbb{F}_{\ell})$. Therefore, via Galois theory $K_{k,\ell}$ is the fixed field of a subgroup of $Gal(\mathbb{\overline{Q}|Q})$ fixing a root of $\tilde{P}_{\varDelta_{k},\ell}$ which corresponds to the stabiliser subgroup of a point in $\mathbb{P}^{1}(\mathbb{F}_{\ell})$. 

For each $(k,\ell)$ in Table \ref{table:polynomials}, the discriminant of the field $K_{k,\ell}$ over $\mathbb{Q}$ is $(-1)^{(\ell-1)/2}\ell^{k+\ell-2}$. This can be shown as follows. One can compute the discriminant $\mathcal{D}(\tilde{P}_{\varDelta_{k},\ell})$ of $\tilde{P}_{\varDelta_{k},\ell}$ and the discriminant $\mathcal{D}_{k,\ell}$ of the number field $K_{k,\ell}$ divides $\mathcal{D}(\tilde{P}_{\varDelta_{k},\ell})$.
Then for each prime divisor $q$ of the discriminant of $\mathcal{D}(\tilde{P}_{\varDelta_{k},\ell})$ one can efficiently compute the power of $q$ that divides the discriminant of $ K_{k,\ell}$ using the algorithms in \cite[Section 6]{Buchmann}. In the cases with $\ell=31$ it is easy to factor the discriminants of the polynomials $\tilde{P}_{\varDelta_{k},31}$ and then the discriminants of $K_{k,31}$ turn out to be $\ell^{k+l-2}$. In the case $\ell=29$, the discriminant of $\tilde{P}_{\varDelta_{k},\ell}$ can be factored as $3^{6}\cdot 19^{4}\cdot 29^{43}\cdot 12653^{2} \cdot 19387^{2}\cdot B^{2}$ where $B$ is a product of big primes and has about 162 decimal digits. We have failed to factor $B$ and checked that it is not divisible by any prime $<10^6$. However, we expect that the prime divisors of $B$ do not divide the discriminant $\mathcal{D}_{16,29}$ of $ K_{16,29}$ and fortunately we can check this with the algorithm of Buchmann-Lenstra without knowing its factorization. In fact, it boils down to the following computation: let $P'$ be the derivative of $\tilde{P}_{\varDelta_{16},29}$, and then in $\mathbb{Z}/B\mathbb{Z}$ we compute $h=\frac{\tilde{P}_{\varDelta_{16},29}}{\mathrm{gcd}(\tilde{P}_{\varDelta_{16},29}, P')}$. Now we take a lift $\tilde{h}$ of $h$. The fact that the minimal polynomial of $\tilde{h}/B$ is a divisor of the resultant R(X) of $(X-\tilde{h}(x))/B$ and $\tilde{P}_{\varDelta_{16},29}(x)$ with respect to $x$ allows us to show that $\tilde{h}/B$ is an algebraic integer. Therefore any prime divisor of $B$ can not be a factor of $\mathcal{D}_{16,29}$. Then it follows from \cite[Section 6]{Buchmann} that $\mathcal{D}_{16,29}=29^{43}$. All the explicit computations involved here are trivial. 

Now each prime $p \neq\ell$ is unramified in $K_{k,\ell}$ and in all four cases it follows that $\tilde{\rho}_{k,\ell}$ is unramified at all $p \neq\ell$. By a lifting of $\tilde{\rho}_{k,\ell}$ we mean a representation $\rho_{k,\ell}: G \rightarrow GL_{2}(\overline{\mathbb F}_{\ell})$ that makes the following diagram commute:
\[\begin{array}{ccc}
G & \stackrel{\tilde{\rho}_{k,\ell}}{\longrightarrow} & PGL_{2}(\mathbb F_\ell) \\
\vcenter{\llap{$\rho_{k,\ell}$}}\Big\downarrow  & & \Big\downarrow \\
GL_{2}(\overline{\mathbb F}_\ell) & \twoheadrightarrow & PGL_{2}(\overline{\mathbb F}_{\ell}) 
\end{array}\]
where the maps on the bottom and the right are the canonical ones. Then from \cite[Section 6]{serre77}, we know $\tilde{\rho}_{k,\ell}$ has a lifting which is unramified outside $\ell$ and therefore has Serre level $1$. By  \cite[Corollary 7.2.10]{book} the minimal weight of a lifting of $\tilde{\rho}_{k,\ell}$ equals $v_{\ell}(Disc(K_{k,\ell}|\mathbb{Q}))-\ell+2=k$. This shows that $\tilde{\rho}_{k,\ell}$ has a lifting $\rho_{k,\ell}$ with weight $k$ and level $1$.

The representation $\rho_{k,\ell}$ is odd in all four cases. Indeed, suppose not. Then the image under $\rho_{k,l}$
of a complex conjugation $\iota$  is $\pm{\rm Id}\in {\rm GL}_2(\mathbb F_\ell)$. This implies that $\tilde\rho_{k,l}(\iota)$ is trivial.
It follows that $K_{\ell,k}$ is totally real. However, this cannot be true. 
Indeed, for  $\ell=31$ the discriminant of $K_{k,\ell}$ is negative, while
for $\ell=29$ the polynomial $\tilde{P}_{\varDelta_{16},29}(x)=\sum_{i=1}^{30} a_i X^i$ has the property that $a_1^2-2a_0a_2<0$, 
which implies that the sum of the reciprocals of its roots is negative.

The fact that Im$\tilde{\rho}_{k,\ell}=PGL_{2}(\mathbb{F}_{\ell})$ implies that $\rho_{k,\ell}$ is absolutely irreducible. For each $(k,\ell)$ in Table \ref{table:polynomials}, the cuspidal space $S_{k}(SL_{2}(\mathbb{Z}))$ has dimension $1$ and Serre's conjecture ensures that $\rho_{k,\ell}\cong \rho_{\varDelta_{k},\ell}$, and hence $\tilde{\rho}_{k,\ell}\cong \tilde{\rho}_{\varDelta_{k},\ell}$. 
\end{proof}

As an example we also computed the following congruence relations in $\mathbb{Z}/31\mathbb{Z}$:

\begin{center}
$\begin{array}{c}
\tau(10^{1000}+4351)=\pm 8, \\
\tau(10^{1000}+10401)= 0, \\
\tau(10^{1000}+11979)=\pm 11, \\
\tau(10^{1000}+17557)=\pm 8. \\
\end{array}$
\end{center}
To obtain these relations, it took about half an hour in SAGE.

In 1947, D. H. Lehmer conjectured that $\tau(n) \neq 0$ for all $n$. In \cite[Theorem 2]{lehmer} he proved that the smallest $n$ for which $\tau(n)=0$ must be a prime. J-P. Serre \cite{serrelehmer} showed that if $\tau(p)=0$ for a prime $p$, then 

\begin{center}
$\begin{array}{lll}
p\equiv -1 \mod 2^{11}3^75^3691, \\
p\equiv  -1,19,31 \mod 7^2 \ \ \ \ \ \ \ \mathrm{and} \\
p\equiv \ a \ non$-$square \mod 23. \\
\end{array}$
\end{center}
We systematically searched for the smallest prime $p$ in these congruence classes for which in addition $\tau(p)\equiv 0\mod 11\cdot 13\cdot 17\cdot 19\cdot 31$. The smallest prime we found is
$$
p = 982149821766199295999.
$$
Then we have
\begin{cor}
The non-vanishing of $\tau(n)$ holds for all 
\[ n<982149821766199295999. \]
\end{cor}
\begin{proof}

\end{proof}

We did the searching computations in PARI and it took around one hour. In \cite[Corollar 7.4]{book}, Bosman's bound is 22798241520242687999 and our bound improves his by a factor approximately equal to $43$. In the paper \cite{zeng}, Zeng also obtained the same prime.

\section*{Acknowledgements}

The author is partly supported by China Scholarship Council (CSC). The author sincerely thanks his advisor Ren\'{e} Schoof who proposed him with this exciting topic, as well as many of his critical suggestions and comments on this paper. The author shows great gratitude to Johan Bosman for his continuous assistance throughout our work. Also, thanks to Bas Edixhoven for his enlightening explanation and comments on the original algorithm. Thanks also go to Mark van Hoeij who explained to the author the algorithm to compute the discriminant of the number field for the case with $(k,\ell)=(16, 29)$.

\end{document}